\documentclass{amsart}
\usepackage[hypertex]{hyperref}
\usepackage{amssymb,amsmath,amsthm}

\newtheorem{theorem}{Theorem}

 \newcommand{\Z}{\ensuremath{\mathbb{Z}}}
 \newcommand{\Q}{\ensuremath{\mathbb{Q}}}
 
 \newcommand{\C}{\ensuremath{\mathbb{C}}}

\newcommand{\ord}{{\rm ord}}

\begin{document}
\title{\bf On the Semigroup of Artin's L-functions holomorphic at $s_0$. II}
\author{Florin Nicolae}
\date{\today}
\address{Simion Stoilow Institute of Mathematics\\ of the
Romanian Academy\\ P.O.BOX 1-764\\
RO-014700 Bucharest, Romania}
\email{Florin.Nicolae@imar.ro}
\begin{abstract}
Let $K/{\Q}$ be a finite Galois extension, and let $s_0\neq 1$ be a complex
number. We present two new criteria for the Artin L-functions to be holomorphic at $s_0$.


\end{abstract}
\maketitle

\indent Let $K/\Q$ be a finite Galois extension with the Galois group $G$, $\chi_1,\ldots,\chi_r$ the irreducible characters of $G$, $d_j=\chi_j(1)$ the dimension of $\chi_j$, $j=1,\ldots,r$, $f_1=L(s, \chi_1, K/\Q),\ldots,f_r=L(s, \chi_r, K/\Q)$ the corresponding Artin $L$-functions, $Ar:=\{f_1^{k_1}\ldots f_r^{k_r}\mid k_1\geq 0,\ldots,k_r\geq 0\}$ the multiplicative semigroup of all Artin $L$-functions. Artin proved (\cite{Ar}, Satz 5, P. 106) that $f_1,\ldots,f_r$ are
multiplicatively independent, that is, 
a relation 
\[f_1^{k_1}\cdot\ldots\cdot f_r^{k_r}=1,\]
where $k_1,\ldots,k_r$ are integers, implies $k_1=\ldots=k_r=0$. This implies that $Ar$ is a 
free semigroup on the generators $f_1,\ldots,f_r.$ For $f,g\in Ar$ we write $f\mid g$ if there exists $h\in Ar$ such that $g=fh$. For 
$s_0\in\C,s_0\neq 1$, let $\mathit{Hol}(s_0)$ be the 
subsemigroup of $Ar$ consisting of the L-functions which are holomorphic at
$s_0$. Artin has conjectured (\cite{Ar}) that every $L$-function is holomophic at $s_0$, that is
$$\mathit{Hol}(s_0)=Ar.$$ I have proved in \cite{Ni}, Theorem 1, that $\mathit{Hol}(s_0)$ is a finitely generated semigroup. The only invertible element in $\mathit{Hol}(s_0)$ is the identity. Every
element of $\mathit{Hol}(s_0)$ is a product of irreducible elements. The set of irreducible elements , denoted by ${\rm
Hilb}(\mathit{Hol}(s_0))$ (the Hilbert basis of $\mathit{Hol}(s_0)$), is finite. By Brauer's induction theorem of characters and by class field
theory, every element of $Ar$ is a quotient of two elements of $\mathit{Hol}(s_0)$. This
implies that the group generated by $\mathit{Hol}(s_0)$ is the free abelian group
$\{f_1^{k_1}\cdot\ldots\cdot f_r^{k_r}\mid k_1,\ldots,k_r\in\Z\}$ of rank $r$.
It follows that the number of elements in ${\rm Hilb}(\mathit{Hol}(s_0))$ is at least
$r$.  The semigroup $\mathit{Hol}(s_0)$ is factorial if and only if the number of elements in ${\rm Hilb}(\mathit{Hol}(s_0))$ is $r$. 
 In \cite{Ni}, Theorem 2, I have proved a necessary and sufficient condition for the holomorphy of the $L$-functions at $s_0$:

\noindent The following assertions are equivalent:\\
i) Artin's conjecture is true: $\mathit{Hol}(s_0)=Ar.$\\
ii) The number of elements in ${\rm Hilb}(\mathit{Hol}(s_0))$ is $r$, and $\prod_{i\in
I}f_i\in \mathit{Hol}(s_0)$ for every subset $I\subset \{1,\ldots,r\}$ with $r-1$
elements.

\noindent Here I prove more: 
\begin{theorem}
The following assertions are equivalent:\\
i) Artin's conjecture is true: $\mathit{Hol}(s_0)=Ar.$\\
ii) The semigroup $\mathit{Hol}(s_0)$ is factorial and for every $k,l\in\{1,\ldots,r\}$, $k\neq l$ there exists $f\in \mathit{Hol}(s_0)$ such that $f_k\mid f$ and $f_l\nmid f$.\\
iii) The semigroup $\mathit{Hol}(s_0)$ is factorial and there exists $1\leq m<r$ such that for every set $M\subseteq \{1,\ldots,r\}$ with $m$ elements and for every $j\in M$ there exists $k_j>0$ such that 
$$\prod_{j\in M}f_j^{k_j}\in \mathit{Hol}(s_0).$$

\end{theorem}
 
\begin{proof}  For a meromorphic function $f$ we denote by $\ord f$ the order of $f$ at $s_0$.
i)$\Rightarrow$ii): If $\mathit{Hol}(s_0)=Ar$ then $\mathit{Hol}(s_0)$ is factorial with the set of irreducibles $\{f_1,\ldots,f_r\}$. For $k,l\in\{1,\ldots,r\}$, $k\neq l$ and $f:=f_k$ we have $f_k\mid f$ and $f_l\nmid f$.\\
ii)$\Rightarrow$i): If Artin's conjecture is not true then there exists $l\in\{1,\ldots,r\}$ such that $\ord f_l<0$. The Dedekind zeta function of $K$ has the decomposition
$$\zeta_K(s)=f_1^{d_1}\cdot\ldots\cdot f_r^{d_r}$$
and is holomorphic at $s_0$, so there exists $k\in\{1,\ldots,r\}$ such that 
$\ord f_k>0$. For $j\in\{1,\ldots,r\}$ let $m_j:=\min\{m\geq 0\mid \ord(f_k^mf_j)\geq 0\}$. The functions $\{f_k^{m_1}f_1,\ldots,f_k^{m_r}f_r\}$ are holomorphic at $s_0$ and are irreducible as elements of the multiplicative semigroup $\mathit{Hol}(s_0)$. By the hypothesis ii) $\mathit{Hol}(s_0)$ is factorial, so the set $\{f_k^{m_1}f_1,\ldots,f_k^{m_r}f_r\}$ is the set of all irreducible elements of $\mathit{Hol}(s_0)$. Since $\ord f_l<0$ we have $m_l>0$. By the hypothesis ii) there exists  $f\in \mathit{Hol}(s_0)$ such that $f_l\mid f$ and $f_k\nmid f$. Since $f$ is a product of irreducible elements and is divisible by $f_l$ it must be divisible by $f_k^{m_l}f_l$, so $f_k\mid f$, a contradiction with $f_k\nmid f$. So Artin's conjecture is true.\\
i)$\Rightarrow$iii): If Artin's conjecture is true then $f_1,\ldots,f_r\in \mathit{Hol}(s_0)$ and iii) is satisfied with $m=1$.\\
iii)$\Rightarrow$ii): Let $k,l\in\{1,\ldots,r\}$, $k\neq l$. Let $M\subseteq \{1,\ldots,r\}$ a set with $m$ elements which contains $k$ and does not contain $l$. By the assumption iii) for every $j\in M$ there exists $k_j>0$ such that  $$f:=\prod_{j\in M}f_j^{k_j}\in \mathit{Hol}(s_0).$$ It holds that $f_k\mid f$ and $f_l\nmid f$.
\end{proof}

\bibliographystyle{plain}

\end{document}